\documentclass[10pt]{article}

\usepackage{amsmath, amsthm, amssymb}
\usepackage{enumitem}
\usepackage{mathtools}
\usepackage{graphicx}
\usepackage[onehalfspacing]{setspace}
\usepackage{etoolbox}
\usepackage[a4paper, margin=3cm]{geometry}
\usepackage{xcolor}
\usepackage[noadjust, nocompress]{cite}
\usepackage{authblk}
\usepackage{nccmath}
\usepackage[normalem]{ulem}
\usepackage{thmtools, hyperref} \hypersetup{colorlinks=true, allcolors=teal}

\theoremstyle{plain}
\newtheorem{theorem}{Theorem}[section]
\newtheorem{lemma}[theorem]{Lemma}
\newtheorem{proposition}[theorem]{Proposition}

\theoremstyle{definition}
\newtheorem{definition}[theorem]{Definition}
\newtheorem{example}[theorem]{Example}
\newtheorem*{remark}{Remark}

\renewenvironment{proof}[1][Proof.]{\begin{trivlist}
		\item[\hskip \labelsep {\bfseries #1}]}{\qed \end{trivlist}}

\setlist{noitemsep, topsep=1ex, parsep=1ex, partopsep=1ex}

\allowdisplaybreaks
\appto\normalsize{
	\abovedisplayskip=2ex plus 1ex minus 1ex
	\belowdisplayskip=2ex plus 1ex minus 1ex
	\abovedisplayshortskip=2ex plus 1ex minus 1ex
	\belowdisplayshortskip=2ex plus 1ex minus 1ex}
\appto\small{
	\abovedisplayskip=2ex plus 1ex minus 1ex
	\belowdisplayskip=2ex plus 1ex minus 1ex
	\abovedisplayshortskip=2ex plus 1ex minus 1ex
	\belowdisplayshortskip=2ex plus 1ex minus 1ex}

\newcommand{\gap}{\vspace{1ex}}

\newcommand{\R}{\mathbb{R}}
\newcommand{\Rn}{\mathbb{R}^n}
\newcommand{\V}{\mathcal{V}}
\newcommand{\W}{\mathcal{W}}

\newcommand{\bfa}{\mathbf{a}}

\newcommand{\bfc}{\mathbf{c}}

\newcommand{\bfe}{\mathbf{e}}

\newcommand{\bfu}{\mathbf{u}}

\newcommand{\bfx}{\mathbf{x}}
\newcommand{\bfy}{\mathbf{y}}
\newcommand{\bfz}{\mathbf{z}}

\newcommand{\tr}{\operatorname{tr}}

\newcommand{\norm}[1]{\left\Vert #1 \right\Vert}
\newcommand{\ip}[2]{\left< #1,\, #2 \right>}
\newcommand{\set}[2]{\left\lbrace #1 \, : \, #2 \right\rbrace}

\title{Generalized convexity of spectral functions on Euclidean Jordan algebras}

\author[1]{Juyoung Jeong}
\affil[1]{\small
	Department of Mathematics\\
	Soongsil University\\
	Seoul 06978\\
	Republic of Korea\\
	jjycjn@ssu.ac.kr}
\date{\today}

\begin{document}

\maketitle

\begin{abstract}
	This paper investigates transfer principles for generalized convexity of spectral functions on Euclidean Jordan algebras. A spectral function is induced by the eigenvalue map and an underlying symmetric function on a symmetric subset of $\Rn$. We establish transfer principles for several generalized convexity notions, such as (strict/semistrict) quasiconvexity and (strict) pseudoconvexity, together with their strong variants, by showing that these properties are preserved in both directions between a spectral function and its associated symmetric function. These results extend the transfer principles for (strict) convexity and provide a unified framework for generalized convexity in the setting of Euclidean Jordan algebras.
	
	\vspace{2ex}
	
\end{abstract}

\vspace{6ex}

\section{Introduction}

Optimization over symmetric cones has emerged as a unifying framework for linear programming, second-order cone programming, and semidefinite programming. A natural algebraic framework for the study of these problems is provided by Euclidean Jordan algebras. This framework allows classical results from matrix analysis to be extended to a broader algebraic setting through an eigenvalue map. One of the central objects of study in this setting is the notion of spectral functions, namely, functions defined on a Euclidean Jordan algebra that depend only on the eigenvalues of their arguments. Formally, a spectral function $G$ on a Euclidean Jordan algebra $\V$ is of the form $G = f \circ \lambda$, where $\lambda : \V \to \Rn$ is the eigenvalue map and $f$ is a symmetric function on $\Rn$. 

\gap
Relationships between the various properties of a spectral function $G$ and those of the underlying symmetric function $f$ are known as \emph{transfer principles}. In the context of convex analysis, these principles are well established. Starting from the pioneering work of Lewis \cite{lewis}, Baes \cite{baes} showed that the convexity of $f$ transfers to $G$. Later, Jeong and Gowda \cite{jeong-gowda_spectralset} observed that the implication is in fact two-sided. Subsequent work of Jeong and Sossa \cite{jeong-sossa_commutation} further proved that strict convexity also transfers between $f$ and $G$ in both directions. This connection has been essential in the design of certain optimization algorithms and in the analysis of the geometry of spectral sets \cite{iusem-seeger, ito-lourenco}.

\gap
However, a large class of optimization problems in economics, engineering, and control theory involve objective functions that are not necessarily convex but satisfy weaker notions of convexity. Generalized convexity, in particular quasiconvexity and pseudoconvexity, plays a pivotal role in modern optimization. For instance, for a continuous quasiconvex function on a convex set, semistrict quasiconvexity is equivalent to the property that every local minimum is global, while pseudoconvexity guarantees that stationary points are globally optimal. Despite their importance, a comprehensive treatment of transfer principles for these generalized notions of convexity within the framework of Euclidean Jordan algebras has remained less developed than in the convex setting.

\gap
In this paper, we address this gap by establishing transfer principles for several generalized convexity notions. We prove that the properties of quasiconvexity, strict quasiconvexity, semistrict quasiconvexity, pseudoconvexity, and strict pseudoconvexity transfer between a spectral function $G$ and its associated symmetric function $f$ in both directions. Furthermore, we extend these results to strongly convex and strongly quasiconvex functions.

\gap
The remainder of the paper is organized as follows. In Section 2, we provide the necessary preliminaries by reviewing basic concepts and results in Euclidean Jordan algebras. We also explore several classes of generalized convex functions. Section 3 presents the main transfer results for these generalized convexity notions, together with results for strong convexity and strong quasiconvexity.

\section{Preliminaries}

Throughout the paper $\Rn$ denotes the $n$-dimensional Euclidean space and $\Rn_+$ denotes the nonnegative orthant of $\Rn$. For $u \in \Rn$, $u^{\downarrow}$ is obtained by rearranging $u$ so that the entries have decreasing order. For $u, v \in \Rn$, we say that $u$ is \emph{majorized by} $v$ and write $u \prec v$ provided, for each $k = 1, 2, \ldots, n-1$,
\[ \sum_{i=1}^{k} u_i^{\downarrow} \leq \sum_{i=1}^{k} v_i^{\downarrow} \quad \text{and} \quad \sum_{i=1}^{n} u_i^{\downarrow} = \sum_{i=1}^{n} v_i^{\downarrow}. \]
Combining Birkhoff's theorem \cite[Theorem 2.A.2]{marshal-olkin-arnold} and Hardy-Little-P{\'o}lya theorem \cite[Theorem 2.B.2]{marshal-olkin-arnold}, for $u, v \in \Rn$, we see that $u \prec v$ if and only if 
\[ u = \sum_{i=1}^{k} \alpha_i P_iv \]
for some $\alpha_1, \alpha_2, \ldots, \alpha_k > 0$ that sum to one and $P_1, P_2, \ldots, P_k$ are permutation matrices. 

\subsection{Euclidean Jordan algebras}

A Euclidean Jordan algebra \cite{faraut-koranyi} is a finite-dimensional real inner product space $(\V, \ip{\cdot}{\cdot})$ endowed with a bilinear operation $\circ : \V \times \V \to \V$, known as the \emph{Jordan product}, satisfying the following three axioms: for all $\bfx, \bfy, \bfz \in \V$,
\begin{itemize}
	\item[$(i)$] $\bfx \circ \bfy = \bfy \circ \bfx$,
	\item[$(ii)$] $\bfx^2 \circ (\bfx \circ \bfy) = \bfx \circ (\bfx^2 \circ \bfy)$, where $\bfx^2 = \bfx \circ \bfx$,
	\item[$(iii)$] $\ip{\bfx \circ \bfy}{\bfz} = \ip{\bfx}{\bfy \circ \bfz}$.
\end{itemize}
The set of squares $\V_+ = \set{\bfx^2}{\bfx \in \V}$ is called the \emph{symmetric cone} of $\V$. This structure generalizes the algebra $\mathcal{S}^n$ of $n \times n$ symmetric matrices with the Jordan product defined as $A \circ B = \frac{1}{2}(AB + BA)$ and the trace inner product $\tr(AB)$ for $A, B \in \mathcal{S}^n$. In this case, $\mathcal{S}^n_+$ is precisely the set of all $n \times n$ positive semidefinite matrices.

\gap
The fundamental classification theorem for Euclidean Jordan algebras asserts that any Euclidean Jordan algebra can be decomposed into a direct sum of simple algebras, which are isomorphic to one of five classes: the algebra of $n \times n$ real symmetric matrices, the algebra of $n \times n$ complex Hermitian matrices, the algebra of $n \times n$ quaternionic Hermitian matrices, a Jordan spin algebra, or an exceptional Jordan algebra corresponding to $3 \times 3$ octonionic Hermitian matrices.

\gap
Throughout the paper, we let $\V$ denote a Euclidean Jordan algebra of rank $n$ with the canonical inner product $\ip{\bfx}{\bfy} = \tr(\bfx \circ \bfy)$. Assume, moreover, that $\V$ has the \emph{unit element} $\bfe \in \V$ such that $\bfe \circ \bfx = \bfx$ for all $\bfx \in \V$. An element $\bfc \in \V$ is called an \emph{idempotent} provided $\bfc^2 = \bfc$. It is a \emph{primitive idempotent} if it cannot be written as a sum of nonzero idempotents. A set $\{\bfe_1, \bfe_2, \ldots, \bfe_n\}$ of orthonormal primitive idempotents that sum to $\bfe$ is then called a \emph{Jordan frame}.

\begin{proposition}[unique spectral decomposition theorem]
	For any $\bfx \in \V$, there exist distinct real numbers $x_1 > x_2 > \cdots > x_p$, where $p \leq n$ depends on $\bfx$, and a system of nonzero orthogonal idempotents $\{\bfc_1, \bfc_2, \ldots, \bfc_p\}$ which sum to $\bfe$ such that
	\[ \bfx = x_1 \bfc_1 + x_2 \bfc_2 + \cdots + x_p \bfc_p. \]
	This decomposition is unique in the sense that if $\bfx = x'_1 \bfc'_1 + x'_2 \bfc'_2 + \cdots + x'_q \bfc'_q$, then $p = q$ and $x_j = x'_j$, $\bfc_j = \bfc'_j$ for all $1 \leq j \leq p$.
\end{proposition}

\begin{proposition}[complete spectral decomposition theorem]
	For any $\bfx \in \V$, there exist real numbers $\lambda_1 \geq \lambda_2 \geq \cdots \geq \lambda_n$ and a Jordan frame $\{\bfe_1, \bfe_2, \ldots, \bfe_n\}$ such that
	\[ \bfx = \lambda_1 \bfe_1 + \lambda_2 \bfe_2 + \cdots + \lambda_n \bfe_n. \]
	Furthermore, if $\bfx = \lambda'_1 \bfe'_1 + \lambda'_2 \bfe'_2 + \cdots + \lambda'_n \bfe'_n$, then $\lambda_j = \lambda'_j$ for all $1 \leq j \leq n$ and $\sum_{\set{j}{\lambda_j = t}} \bfe_j = \sum_{\set{j}{\lambda'_j = t}} \bfe'_j$ for all $t \in \R$.
\end{proposition}

The map $\lambda : \V \to \Rn$ defined by $\bfx \mapsto \lambda(\bfx) = (\lambda_1, \lambda_2, \ldots, \lambda_n) \in \R^n$ is called the \emph{eigenvalue map}. With the eigenvalue map, one can define the \emph{trace} and the \emph{determinant} of $\bfx \in \V$, respectively, by 
\[ \tr(\bfx) = \sum_{i=1}^{n} \lambda_i(\bfx) \quad \text{and} \quad \det(\bfx) = \prod_{i=1}^{n} \lambda_i(\bfx). \]
Also, it can be shown that $\V_+ = \set{\bfx \in \V}{\lambda(\bfx) \in \Rn_+}$. 

\gap
For $\bfx, \bfy \in \V$, we say that $\bfx$ and $\bfy$ \emph{operator commute} provided $L_{\bfx} L_{\bfy} = L_{\bfy} L_{\bfx}$, where $L_{\bfx} : \V \to \V$ is defined by $L_{\bfx}(\bfa) = \bfa \circ \bfx$ for all $\bfa \in \V$. It is well known that operator commutativity is equivalent to a simultaneous spectral decomposition. That is, there exist real numbers $\lambda_1 \geq \lambda_2 \geq \cdots \geq \lambda_n$ and $\mu_1 \geq \mu_2 \geq \cdots \geq \mu_n$ and a Jordan frame $\{\bfe_1, \bfe_2, \ldots, \bfe_n\}$ such that
\[ \bfx = \lambda_1 \bfe_1 + \lambda_2 \bfe_2 + \cdots + \lambda_n \bfe_n \quad \text{and} \quad \bfy = \mu_{\sigma(1)} \bfe_1 + \mu_{\sigma(2)} \bfe_2 + \cdots + \mu_{\sigma(n)} \bfe_n, \]
for some permutation $\sigma$ of $\{1, 2, \ldots, n\}$.

\begin{proposition} \label{prop:vonNeumann_inequality}
	For $\bfx, \bfy \in \V$, the following von Neumann-type inequality holds,
	\[ \ip{\bfx}{\bfy} \leq \ip{\lambda(\bfx)}{\lambda(\bfy)}. \]
\end{proposition}

If equality holds in the von Neumann-type inequality, that is, when $\ip{\bfx}{\bfy} = \ip{\lambda(\bfx)}{\lambda(\bfy)}$, we say that $\bfx$ and $\bfy$ \emph{strongly operator commute}. Interestingly enough, strong operator commutativity is equivalent to simultaneous ordered spectral decomposition, i.e., $\bfx$ and $\bfy$ have simultaneous spectral decomposition with an additional condition that the permutation $\sigma$ is an identity permutation. More precisely, there exist real numbers $\lambda_1 \geq \lambda_2 \geq \cdots \geq \lambda_n$ and $\mu_1 \geq \mu_2 \geq \cdots \geq \mu_n$ and a Jordan frame $\{\bfe_1, \bfe_2, \ldots, \bfe_n\}$ such that
\[ \bfx = \lambda_1 \bfe_1 + \lambda_2 \bfe_2 + \cdots + \lambda_n \bfe_n \quad \text{and} \quad \bfy = \mu_1 \bfe_1 + \mu_2 \bfe_2 + \cdots + \mu_n \bfe_n. \]

\gap
There are several equivalent descriptions of strong operator commutativity which will be frequently used in this paper. Recall that for any $\bfx, \bfy \in \V$, the following majorization inequality holds:
\[ \lambda(\bfx + \bfy) \prec \lambda(\bfx) + \lambda(\bfy). \]

\begin{proposition} \label{prop:strongly_operator_commute}
	For $\bfx, \bfy \in \V$, the following are equivalent:
	\begin{itemize}
		\item[$(a)$] $\ip{\bfx}{\bfy} = \ip{\lambda(\bfx)}{\lambda(\bfy)}$,
		\item[$(b)$] $\lambda(\bfx + \bfy) = \lambda(\bfx) + \lambda(\bfy)$,
		\item[$(c)$] $\norm{\lambda(\bfx) - \lambda(\bfy)} = \norm{\bfx - \bfy}$.
	\end{itemize}
\end{proposition}

The eigenvalue map $\lambda : \V \to \Rn$ allows us to define spectral sets and functions. A set $C \subseteq \V$ is called a \emph{spectral set} if there exists a symmetric set $Q \subseteq \Rn$ such that $C = \lambda^{-1}(Q)$. As $\V_+ = \lambda^{-1}(\Rn_+)$, it is a primary example of spectral sets. A function $G : C \to \R$, where $C = \lambda^{-1}(Q)$, is called a \emph{spectral function} if there exists a symmetric function $f : Q \to \R$ such that $G = f \circ \lambda$. Recall that $\tr$ and $\det$ are spectral functions. 

\gap
Various topological and convexity properties of $Q$ can be transferred to $C = \lambda^{-1}(Q)$; see \cite{jeong-gowda_spectralset, jeong-gowda_spectralcone, gowda-jeong_connectedness}. Likewise, many analytic and convexity properties of $f : \Rn \to \R$ can be transferred to $G = f \circ \lambda$; see \cite{baes, jeong-gowda_spectralset, lourenco-takeda}. In particular, we record below a transfer principle for Fr{\'e}chet differentiability from \cite[Theorem 38]{baes}.

\begin{proposition} \label{prop:transfer_differentiability}
	Let $Q \subseteq \Rn$ be an open symmetric set and $f : Q \to \R$ be a symmetric function. Define a spectral set $C = \lambda^{-1}(Q)$ and a spectral function $G : C \to \R$ by $G = f \circ \lambda$. Let $\bfx = \sum_{i=1}^{n} \lambda_i \bfe_i \in C$, where $\{e_1, e_2, \ldots, e_n\}$ is a Jordan frame associated with a complete spectral decomposition of $x$. If $f$ is differentiable at $\lambda(\bfx)$, then $G$ is differentiable at $\bfx$ and
	\[ \nabla G(\bfx) = \sum_{i=1}^{n} \big[ \nabla f(\lambda(\bfx)) \big]_i \bfe_i. \]
	Moreover, $\bfx$ and $\nabla G(\bfx)$ operator commute.
\end{proposition}

\subsection{Convexity and its generalizations}

In this subsection, we summarize the definitions and basic properties of convexity and its generalizations. We also describe how these notions are related. Detailed treatments of these subjects can be found, for example, in \cite{cambini-martein}.

\begin{definition}[Convexity] \label{def:convexity}
	Let $C$ be a convex set in $\V$ and consider a function $G : C \to \R$.
	\begin{itemize}
		\item[(a)] $G$ is said to be \emph{convex} on $C$ if 
		\[ G((1 - \alpha)\bfx + \alpha \bfy) \leq (1 - \alpha)G(\bfx) + \alpha G(\bfy) \]
		for all $\bfx, \bfy \in C$ and $\alpha \in [0, 1]$.
		
		\item[(b)] $G$ is said to be \emph{strictly convex} on $C$ if 
		\[ G((1 - \alpha)\bfx + \alpha \bfy) < (1 - \alpha)G(\bfx) + \alpha G(\bfy) \]
		for all $\bfx, \bfy \in C$ with $\bfx \neq \bfy$ and $\alpha \in (0, 1)$.
	\end{itemize}
\end{definition}	

When $G$ is convex, every local minimizer of $G$ is a global minimizer. If, in addition, $C$ is open and $G$ is Fréchet differentiable on $C$, then every critical point of $G$ is a global minimizer. In the case that $G$ is strictly convex, it has at most one global minimizer.

\gap
$G$ is convex if and only if the epigraph
\[ \operatorname{epi}(G) = \set{(\bfx, \alpha) \in C \times \R}{G(\bfx) \leq \alpha} \]
is a convex set in $\V \times \R$. Furthermore, assume that $C$ is open and $G$ is Fr{\'e}chet differentiable on $C$. Then $G$ is convex on $C$ if and only if
\[ \ip{\nabla G(\bfx)}{\bfy - \bfx} \leq G(\bfy) - G(\bfx) \]
for all $\bfx, \bfy \in C$.

\gap
The concept of convexity is closely related to majorization. A function $f : \Rn \to \R$ is said to be \emph{Schur-convex} provided $u \prec v$ implies $f(u) \leq f(v)$ for all $u, v \in \Rn$. It is said to be \emph{strictly Schur-convex} provided $u \prec v$ with $u^{\downarrow} \neq v^{\downarrow}$ implies $f(u) < f(v)$. It is well known that every symmetric convex function is Schur-convex; see \cite[Proposition 3.C.2]{marshal-olkin-arnold}.

\begin{definition}[Quasiconvexity] \label{def:quasiconvexity}
	Let $C$ be a convex set in $\V$ and consider a function $G : C \to \R$.
	\begin{itemize}
		\item[(a)] $G$ is said to be \emph{quasiconvex} on $C$ if the inequality
		\[ G((1 - \alpha)\bfx + \alpha \bfy) \leq \max\{ G(\bfx), G(\bfy) \} \]
		holds for all $\bfx, \bfy \in C$ and $\alpha \in [0, 1]$.
		
		\item[(b)] $G$ is said to be \emph{strictly quasiconvex} on $C$ if the strict inequality
		\[ G((1 - \alpha)\bfx + \alpha \bfy) < \max\{ G(\bfx), G(\bfy) \} \]
		holds for all $\bfx, \bfy \in C$ with $\bfx \neq \bfy$ and $\alpha \in (0, 1)$.
		
		\item[(c)] $G$ is said to be \emph{semistrictly quasiconvex} on $C$ if the strict inequality
		\[ G((1 - \alpha)\bfx + \alpha \bfy) < \max\{ G(\bfx), G(\bfy) \} \]
		holds for all $\bfx, \bfy \in C$ with $G(\bfx) \neq G(\bfy)$ and $\alpha \in (0, 1)$.
	\end{itemize}
\end{definition}

Let $G$ be a continuous quasiconvex function on $C$. Then $G$ is semistrictly quasiconvex if and only if every local minimum of $G$ is a global minimum \cite[Theorem 2.3.5]{cambini-martein}. We now describe a few equivalent characterizations of quasiconvexity. First, $G$ is quasiconvex if and only if a $\beta$-sublevel set
\[ L_{\beta}^{-}(G) = \set{\bfx \in C}{G(\bfx) \leq \beta} \]
is a convex set in $\V$ for every $\beta \in \R$. Also, in the case that $C$ is an open convex set in $\V$ and $G$ is Fr{\'e}chet differentiable on $C$, $G$ is quasiconvex if and only if
\[ G(\bfy) \leq G(\bfx) \implies \ip{\nabla G(\bfx)}{\bfy - \bfx} \leq 0 \]
for all $\bfx, \bfy \in C$. Likewise, under the additional assumption that $\nabla G(\bfx) \neq \mathbf{0}$ for all $\bfx \in C$, $G$ is strictly quasiconvex if and only if
\[ G(\bfy) \leq G(\bfx) \implies \ip{\nabla G(\bfx)}{\bfy - \bfx} < 0 \]
for all $\bfx, \bfy \in C$ with $\bfx \neq \bfy$. This characterization motivates the following definition of pseudoconvexity.

\begin{definition}[Pseudoconvexity] \label{def:pseudoconvexity}
	Let $C$ be an open convex set in $\V$ and consider a Fr{\'e}chet differentiable function $G : C \to \R$.
	\begin{itemize}
		\item[(a)] $G$ is said to be \emph{pseudoconvex} on $C$ if the implication
		\[ G(\bfy) < G(\bfx) \implies \ip{\nabla G(\bfx)}{\bfy - \bfx} < 0 \]
		holds for all $\bfx, \bfy \in C$.
		
		\item[(b)] $G$ is said to be \emph{strictly pseudoconvex} on $C$ if the implication
		\[ G(\bfy) \leq G(\bfx) \implies \ip{\nabla G(\bfx)}{\bfy - \bfx} < 0 \]
		holds for all $\bfx, \bfy \in C$ with $\bfx \neq \bfy$.
	\end{itemize}
\end{definition}

If $G$ is pseudoconvex, then every critical point is a global minimum. Indeed, with the assumption that $G$ is quasiconvex, $G$ is pseudoconvex if and only if every critical point is a global minimum \cite[Theorems 3.2.5 and 3.2.9]{cambini-martein}. Furthermore, there exists at most one critical point when $G$ is strictly pseudoconvex. 

\gap
The diagram below summarizes the implications among notions of convexity.
\begin{center}
	\includegraphics[scale=0.667]{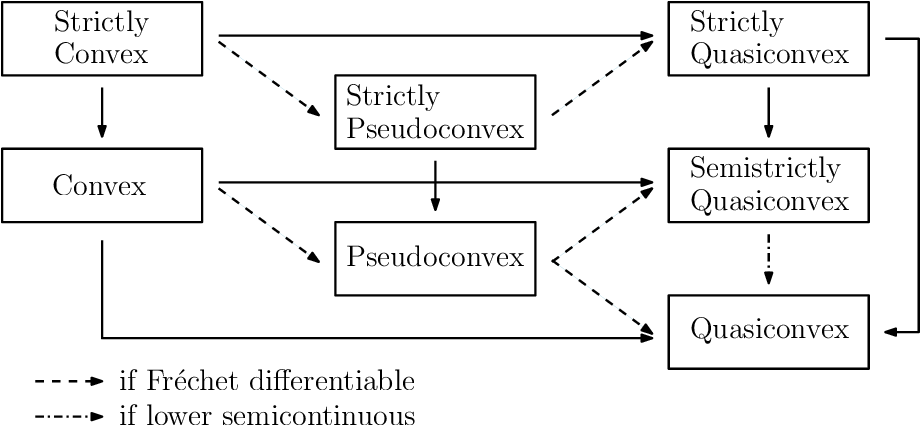}
\end{center}

\section{Transfer principles for generalized convexity}

In this section, we describe transfer principles for generalized convexity.

\begin{theorem} \label{thm:convexity}
	Let $C = \lambda^{-1}(Q)$ be a spectral convex set in $\V$ and $G : C \to \R$ be a spectral function with the associated symmetric function $f : Q \to \R$. 
	\begin{itemize}
		\item[$(i)$] $G$ is convex on $C$ if and only if $f$ is convex on $Q$.
		\item[$(ii)$] $G$ is strictly convex on $C$ if and only if $f$ is strictly convex on $Q$.
	\end{itemize}
\end{theorem}

\begin{proof}
	In \cite[Theorem 41]{baes}, Baes observed that the convexity of $f$ can be transferred to $G$. Later, in \cite[Theorem 6]{jeong-gowda_spectralset}, Jeong and Gowda demonstrated that convexity can indeed be transferred between $f$ and $G$ in both directions. The (two-sided) transfer principle for strict convexity has been observed and proved in \cite[Proposition 6]{jeong-sossa_commutation}.
\end{proof}

The main idea for verifying the above theorem is the fact that symmetric (strictly) convex functions are (strictly) Schur-convex. Fortunately, an analogous statement remains valid for symmetric (strictly) quasiconvex functions.

\begin{lemma} \label{lem:quasiconvex_function}
	Let $Q$ be a convex and symmetric set in $\Rn$, $f : Q \to \R$ be a symmetric function, and $u, v \in Q$. If $f$ is quasiconvex and $u \prec v$, then $f(u) \leq f(v)$. Moreover, if $f$ is strictly quasiconvex and $u \prec v$ with $u^{\downarrow} \neq v^{\downarrow}$, then $f(u) < f(v)$.
\end{lemma}

\begin{proof}
	The proof of the quasiconvexity case can be found in \cite[Theorem 3.C.3]{marshal-olkin-arnold}. We include our own proof here for completeness. Let $u, v \in Q$ such that $u \prec v$. By Hardy–Littlewood–P{\'o}lya theorem and Birkhoff's theorem, we may write
	\[ u = \sum_{i=1}^{k} \alpha_i P_i v \]
	for some $\alpha_1, \alpha_2, \ldots, \alpha_k > 0$ summing to one and permutation matrices $P_1, P_2, \ldots, P_k$. First, suppose $f$ is symmetric and quasiconvex. Then
	\[ f(u) = f \bigg( \sum_{i=1}^{k} \alpha_i P_i v \bigg) \leq \max_{1 \leq i \leq k} \big\{ f(P_i v) \big\} = \max_{1 \leq i \leq k} \big\{ f(v) \big\} = f(v). \]
	
	For the second assertion, suppose further that $f$ is strictly quasiconvex and that $u^\downarrow \neq v^\downarrow$. Since $u^\downarrow \neq v^\downarrow$, we must have $k \geq 2$. Define
	\[ w = \sum_{i = 2}^k \frac{\alpha_i}{1 - \alpha_1} P_i v. \]
	Because $Q$ is convex and symmetric, we have $w \in Q$. Moreover, $w \prec v$. If $P_1v = w$, then
	\[ u = \alpha_1 P_1 v + (1 - \alpha_1) w = \alpha_1 P_1 v + (1 - \alpha_1) P_1 v = P_1 v, \]
	which implies $u^\downarrow = v^\downarrow$, a contradiction. Hence,
	$P_1 v \neq w$. Then strict quasiconvexity of $f$ implies
	\[ f(u) = f \big( \alpha_1 P_1 v + (1 - \alpha_1) w \big) < \max \{f(P_1v), f(w)\}. \]
	Since $f$ is symmetric, $f(P_1v) = f(v)$. Also, as strict quasiconvexity
	implies quasiconvexity, the first assertion of the lemma, applied to
	$w \prec v$, yields $f(w) \leq f(v)$. Thus,
	\[ f(u) < \max \{f(P_1v), f(w)\} \leq \max \{f(v), f(v)\} = f(v). \]
	This completes the proof.
\end{proof}

\begin{theorem} \label{thm:quasiconvexity}
	Let $Q$ be a convex symmetric set in $\Rn$ and $f : Q \to \R$ be a symmetric function. Define $C = \lambda^{-1}(Q)$ and $G : C \to \R$ by $G = f \circ \lambda$. Then the following hold:
	\begin{itemize}
		\item[$(i)$] $G$ is quasiconvex on $C$ if and only if $f$ is quasiconvex on $Q$.
		\item[$(ii)$] $G$ is strictly quasiconvex on $C$ if and only if $f$ is strictly quasiconvex on $Q$.
		\item[$(iii)$] Suppose, in addition, that either $G$ or $f$ is quasiconvex. Then $G$ is semistrictly quasiconvex on $C$ if and only if $f$ is semistrictly quasiconvex on $Q$.
	\end{itemize}
\end{theorem}

\begin{proof}
	For any $\bfx, \bfy \in C$ and $\alpha \in [0, 1]$, Ky-Fan inequality and positive homogeneity of $\lambda$ imply that
	\[ \lambda((1 - \alpha)\bfx + \alpha \bfy) \prec \lambda((1 - \alpha) \bfx) + \lambda(\alpha \bfy) = (1 - \alpha) \lambda(\bfx) + \alpha \lambda(\bfy). \]
	
	\gap
	
	$(i)$ It has been observed \cite[Theorem 41]{baes} that quasiconvexity of $f$ is transferred to $G$. For completeness, we include a proof using majorization techniques. Suppose $f$ is quasiconvex, $\bfx, \bfy \in C$, and $\alpha \in [0, 1]$. Then
	\begin{subequations} \label{eq:quasiconvex_inequality}
		\begin{align}
			G \big( (1 - \alpha)\bfx + \alpha \bfy \big) 
			& = f \big( \lambda( (1 - \alpha) \bfx + \alpha \bfy) \big) \notag \\ 
			& \leq f \big( (1 - \alpha)\lambda(\bfx) + \alpha  \lambda(\bfy) \big) \label{eq:quasiconvex_inequality_1} \\ 
			& \leq \max \big\{ f \big( \lambda(\bfx) \big), f \big( \lambda(\bfy) \big) \big\} \label{eq:quasiconvex_inequality_2} \\ 
			& = \max \big\{ G(\bfx), G(\bfy) \big\}, \notag 
		\end{align}
	\end{subequations}
	where the inequality \eqref{eq:quasiconvex_inequality_1} is due to \autoref{lem:quasiconvex_function} and the inequality \eqref{eq:quasiconvex_inequality_2} is due to quasiconvexity of $f$. This shows that $G$ is quasiconvex.
	
	Conversely, assume that $G$ is quasiconvex. Take $u, v \in Q$ and $\alpha \in [0, 1]$ arbitrary. Write $w = (1 - \alpha) u + \alpha v \in Q$. Fix a Jordan frame $\{\bfe_1, \bfe_2, \ldots, \bfe_n\}$ and define $\bfx, \bfy, \bfz \in C$ as follows
	\begin{equation} \label{eq:quasiconvex_xyz}
		\begin{aligned} 
			\bfx & = \sum_{i=1}^{n} u_i \bfe_i, \quad \bfy = \sum_{i=1}^{n} v_i \bfe_i, \\
			\bfz & = \sum_{i=1}^{n} w_i \bfe_i = \sum_{i=1}^{n} \big[ (1 - \alpha) u_i + \alpha v_i \big] \bfe_i = (1 - \alpha) \bfx + \alpha \bfy.
		\end{aligned}
	\end{equation}
	Note that $\lambda(\bfx) = u^{\downarrow}$, $\lambda(\bfy) = v^{\downarrow}$, and $\lambda(\bfz) = w^{\downarrow}$. Since $f$ is symmetric,
	\begin{subequations} \label{eq:quasiconvex_inequality_converse}
		\begin{align}
			f(w) & = f(w^{\downarrow}) = G(\bfz) \notag \\
			& \leq \max \big\{ G(\bfx), G(\bfy) \big\} \label{eq:quasiconvex_inequality_converse_1} \\
			& = \max \big\{ f(u^{\downarrow}), f(v^{\downarrow}) \big\} = \max \big\{ f(u), f(v) \big\}. \notag 
		\end{align}
	\end{subequations}
	Thus, $f$ is quasiconvex.
	
	\gap
	
	$(ii)$ Suppose $f$ is strictly quasiconvex, $\bfx, \bfy \in C$ with $\bfx \neq \bfy$, and $\alpha \in (0, 1)$. If $\lambda(\bfx) \neq \lambda(\bfy)$, the strict quasiconvexity of $f$ forces the inequality \eqref{eq:quasiconvex_inequality_2} to be strict. On the other hand, if $\lambda(\bfx) = \lambda(\bfy)$, we see that $\bfx$ and $\bfy$ cannot strongly operator commute because otherwise we have $0 = \norm{\lambda(\bfx) - \lambda(\bfy)} = \norm{\bfx - \bfy}$ from \autoref{prop:strongly_operator_commute}, which contradict the assumption that $\bfx \neq \bfy$. Thus, $\bfx$ and $\bfy$ as well as $(1 - \alpha)\bfx$ and $\alpha \bfy$ do not strongly operator commute. Then
	\[ \lambda \big( (1 - \alpha)\bfx + \alpha \bfy \big) \neq \lambda \big( (1 - \alpha) \bfx \big) + \lambda(\alpha \bfy) = (1 - \alpha)\lambda(\bfx) + \alpha \lambda(\bfy). \]
	Hence, by \autoref{lem:quasiconvex_function} applied to $u = \lambda \big( (1 - \alpha)\bfx + \alpha \bfy \big)$ and $v = (1 - \alpha)\lambda(\bfx) + \alpha \lambda(\bfy)$, the inequality \eqref{eq:quasiconvex_inequality_1} becomes strict.
	
	For the converse, assume $G$ is strictly quasiconvex. Take $u, v \in Q$ with $u \neq v$, and $\alpha \in (0, 1)$. Define $\bfx, \bfy, \bfz \in \V$ as in \eqref{eq:quasiconvex_xyz}. Since $\bfx \neq \bfy$ as $u \neq v$, we have strict inequality in \eqref{eq:quasiconvex_inequality_converse_1}, implying that $f$ is strictly quasiconvex.
	
	\gap
	
	$(iii)$ By part $(i)$, the additional assumption implies that both $f$ and $G$ are quasiconvex. 
	
	Suppose that $f$ is semistrictly quasiconvex. Now take $\bfx, \bfy \in C$ with $G(\bfx) < G(\bfy)$ and $\alpha \in (0, 1)$. Since $f(\lambda(\bfx)) < f(\lambda(\bfy))$, semistrict quasiconvexity of $f$ ensures that the inequality \eqref{eq:quasiconvex_inequality_2} is strict.
	
	To see the converse, assume that $G$ is semistrictly quasiconvex. Let $u, v \in Q$ be such that $f(u) < f(v)$, and let $\alpha \in (0, 1)$. Define $\bfx, \bfy, \bfz \in \V$ as in \eqref{eq:quasiconvex_xyz}. Then
	\begin{equation} \label{eq:transfer_inequality}
		G(\bfx) = f(\lambda(\bfx)) = f(u^{\downarrow}) = f(u) < f(v) = f(v^{\downarrow}) = f(\lambda(\bfy)) = G(\bfy)
	\end{equation}
	It follows that the inequality \eqref{eq:quasiconvex_inequality_converse_1} is strict; thus $f$ is semistrictly quasiconvex.
\end{proof}

\begin{example}
	Consider the spectral convex set $C = \operatorname{int}(\V_+)$ and the function $G : C \to \R$ given by
	\[ G(\bfx) = \frac{\lambda_{\max}(\bfx)}{\lambda_{\min}(\bfx)}. \]
	This function is known as the \emph{spectral condition number} in the literature \cite{massey-rios-sossa}. Also, $G$ is a spectral function with its corresponding symmetric function $f : \operatorname{int}(\Rn_+) \to \R$ given by $f(u) = \frac{\max_i u_i}{\min_i u_i}$. Since $f$ is quasiconvex, the transfer principle for quasiconvexity verifies that $G$ is quasiconvex on $C$.
\end{example}

\begin{remark}
	In the special case $Q = \Rn$, one may also prove the transfer principle for quasiconvexity via a sublevel set description. For $f : \Rn \to \R$ and $G : \V \to \R$ defined by $G = f \circ \lambda$, it is readily seen that
	\[ L_{\beta}^{-}(G) = \lambda^{-1} \big( L_{\beta}^{-}(f) \big). \]
	It follows that the sublevel set $L_{\beta}^{-}(G)$ is a spectral set with its associated symmetric set $L_{\beta}^{-}(f)$. Now, from the transfer principle for convexity, we have
	\begin{align*}
		\text{$G$ is a quasiconvex function on $\V$} & \iff \text{$L_{\beta}^{-}(G)$ is a convex set in $\V$} \\
		& \iff \text{$L_{\beta}^{-}(f)$ is a convex set in $\Rn$} \\
		& \iff \text{$f$ is a quasiconvex function on $\Rn$}
	\end{align*}
	This proves the transfer principle for quasiconvexity.
\end{remark}

Let $Q$ be an open symmetric set and $f : Q \to \R$ be a symmetric function. If $f$ is differentiable at $u \in Q$, then the following implication holds; see \cite[Remark 37]{baes}.
\begin{equation} \label{eq:gradient_equality_forward}
	u_i = u_j \implies [\nabla f(u)]_i = [\nabla f(u)]_j.
\end{equation}

We first record a modification of \autoref{prop:transfer_differentiability} under the (quasi)convexity assumption.

\begin{lemma} \label{lem:gradient_order_quasiconvex}
	Let $Q \subseteq \Rn$ be an open convex symmetric set and $f : Q \to \R$ be a differentiable symmetric quasiconvex function. Define a spectral set $C = \lambda^{-1}(Q)$ and a spectral function $G : C \to \R$ by $G = f \circ \lambda$. Then, for every $\bfx \in C$, $\bfx$ and $\nabla G(\bfx)$ strongly operator commute. In particular,
	\[ \lambda(\nabla G(\bfx)) = \nabla f(\lambda(\bfx)). \]
\end{lemma}

\begin{proof}
	Let $\bfx = \sum_{i=1}^{n} \lambda_i \bfe_i$, where $\lambda_1 \geq \lambda_2 \geq \cdots \geq \lambda_n$. By \autoref{prop:transfer_differentiability},
	\[ \nabla G(\bfx) = \sum_{i=1}^{n} [\nabla f(\lambda(\bfx))]_i \bfe_i. \]
	Thus it remains to show that the components of $\nabla f(\lambda(\bfx))$ are arranged in decreasing order. Write $u = \lambda(\bfx)$. If $u_i = u_j$, then $[\nabla f(u)]_i = [\nabla f(u)]_j$ from \eqref{eq:gradient_equality_forward}. If $u_i > u_j$, let $v$ be obtained from $u$ by swapping the $i$th and $j$th entries and keeping the remaining entries the same. By symmetry of $f$, $f(v) = f(u)$. Since $f$ is differentiable and quasiconvex,
	\begin{align*}
		0 \geq \ip{\nabla f(u)}{v - u} 
		& = [\nabla f(u)]_i (v_i - u_i) + [\nabla f(u)]_j (v_j - u_j) \\
		& = \big( [\nabla f(u)]_i - [\nabla f(u)]_j \big) \big( u_j - u_i \big)
	\end{align*}
	Since $u_j - u_i < 0$, we obtain $[\nabla f(u)]_i \geq [\nabla f(u)]_j$. Thus the entries of $\nabla f(u)$ are in decreasing order, and the displayed decomposition of $\nabla G(\bfx)$ is an ordered simultaneous spectral decomposition with $\bfx$. Hence the result follows.
\end{proof}

The following lemma shows that the converse implication of \eqref{eq:gradient_equality_forward} holds when $f$ is assumed, in addition, to be strictly pseudoconvex.

\begin{lemma} \label{lem:pseudoconvex_lemma_1}
	Let $Q$ be an open convex symmetric set in $\R^n$ and $f : Q \to \R$ be a symmetric strictly pseudoconvex function. Then, for every $u \in Q$,
	\begin{equation} \label{eq:gradient_equality_backward}
		[\nabla f(u)]_i = [\nabla f(u)]_j \implies u_i = u_j.
	\end{equation}
\end{lemma}

\begin{proof}
	Assume, to the contrary, that $[\nabla f(u)]_i = [\nabla f(u)]_j$ but $u_i \neq u_j$. Let $v$ be the vector obtained from $u$ by swapping the $i$th and $j$th entries and keeping the remaining entries the same. By the symmetry of $f$, $f(u) = f(v)$. Then
	\begin{align*}
		\ip{\nabla f(u)}{v - u} & = [\nabla f(u)]_i (v_i - u_i) + [\nabla f(u)]_j (v_j - u_j) \\
		& = [\nabla f(u)]_i (u_j - u_i) + [\nabla f(u)]_i (u_i - u_j) = 0
	\end{align*}
	This means that $f$ is not strictly pseudoconvex, a contradiction. Therefore, we have $u_i = u_j$ whenever $[\nabla f(u)]_i = [\nabla f(u)]_j$.
\end{proof}

\begin{lemma} \label{lem:pseudoconvex_lemma_2}
	For $\bfx, \bfy, \bfz \in \V$, suppose $\bfx$ and $\bfy$ strongly operator commute and $\bfy$ and $\bfz$ strongly operator commute. If
	\begin{equation} \label{eq:transitivity_assumption}
		\lambda_i = \lambda_j \iff \mu_i = \mu_j \quad \forall i, j \in \{1, 2, \ldots, n\},
	\end{equation}
	where $\lambda_1 \geq \lambda_2 \geq \cdots \geq \lambda_n$ and $\mu_1 \geq \mu_2 \geq \cdots \geq \mu_n$ are the eigenvalues of $\bfx$ and $\bfy$, respectively, then $\bfx$ and $\bfz$ also strongly operator commute. 
\end{lemma}

\begin{proof}
	Since $\bfx$ and $\bfy$ strongly operator commute, there exist a common Jordan frame $\{\bfe_1, \bfe_2, \ldots, \bfe_n\}$ such that
	\[ \bfx = \lambda_1 \bfe_1 + \lambda_2 \bfe_2 + \cdots + \lambda_n \bfe_n \quad \text{and} \quad \bfy = \mu_1 \bfe_1 + \mu_2 \bfe_2 + \cdots + \mu_n \bfe_n. \]
	Let $\alpha_1 > \alpha_2 > \cdots > \alpha_p$ and $\beta_1 > \beta_2 > \cdots > \beta_p$ be distinct eigenvalues of $\bfx$ and $\bfy$, respectively. Note that they have the same number of distinct eigenvalues by the assumption \eqref{eq:transitivity_assumption}. Define $I_j = \set{1 \leq i \leq n}{\lambda_i = \alpha_j}$ and $\bfc_j = \sum_{i \in I_j} \bfe_i$ for $1 \leq j \leq p$. Since both eigenvalue sequences are nonincreasing and satisfy \eqref{eq:transitivity_assumption}, $I_j = \set{1 \leq i \leq n}{\mu_i = \beta_j}$ for each $j$. Consequently, we obtain the unique spectral decomposition of $\bfx$ and $\bfy$ as follows.
	\[ \bfx = \alpha_1 \bfc_1 + \alpha_2 \bfc_2 + \cdots + \alpha_p \bfc_p \quad \text{and} \quad \bfy = \beta_1 \bfc_1 + \beta_2 \bfc_2 + \cdots + \beta_p \bfc_p. \]
	Now, since $\bfy$ and $\bfz$ strongly operator commute, there exist real numbers $\nu_1 \geq \nu_2 \geq \cdots \geq \nu_n$ and a Jordan frame $\{\bfe'_1, \bfe'_2, \ldots, \bfe'_n\}$ such that
	\[ \bfy = \mu_1 \bfe'_1 + \mu_2 \bfe'_2 + \cdots + \mu_n \bfe'_n \quad \text{and} \quad \bfz = \nu_1 \bfe'_1 + \nu_2 \bfe'_2 + \cdots + \nu_n \bfe'_n. \]
	By the uniqueness of the spectral decomposition, we have 
	\[ \bfc_j = \sum_{i \in I_j} \bfe_i = \sum_{i \in I_j} \bfe'_i \]
	for all $1 \leq j \leq p$. It follows that
	\[ \bfx = \sum_{j=1}^{p} \alpha_j \bfc_j = \sum_{j=1}^{p} \alpha_j \bigg( \sum_{i \in I_j} \bfe'_i \bigg) = \sum_{j=1}^{p} \sum_{i \in I_j} \alpha_j \bfe'_i = \sum_{j=1}^{p} \sum_{i \in I_j} \lambda_i \bfe'_i = \sum_{i=1}^{n} \lambda_i \bfe'_i. \]
	Hence, $\bfx$ and $\bfz$ have simultaneous ordered spectral decomposition and therefore strongly operator commute.
\end{proof}

\begin{theorem} \label{thm:pseudoconvexity}
	Let $Q$ be an open convex symmetric set in $\Rn$ and $f : Q \to \R$ be a Fr{\'e}chet differentiable symmetric function. Define $C = \lambda^{-1}(Q)$ and $G : C \to \R$ by $G = f \circ \lambda$. Then the following hold:
	\begin{itemize}
		\item[$(i)$] $G$ is pseudoconvex on $C$ if and only if $f$ is pseudoconvex on $Q$.
		\item[$(ii)$] $G$ is strictly pseudoconvex on $C$ if and only if $f$ is strictly pseudoconvex on $Q$.
	\end{itemize}
\end{theorem}

\begin{proof}
	We first record a useful inequality. Suppose that $f$ is quasiconvex, and let $\bfx, \bfy \in C$ satisfy $G(\bfy) \leq G(\bfx)$. By \autoref{lem:gradient_order_quasiconvex}, $\bfx$ and $\nabla G(\bfx)$ strongly operator commute and $\lambda(\nabla G(\bfx)) = \nabla f(\lambda(\bfx))$. Thus
	\begin{subequations} \label{pseudoconvex_inequality}
		\begin{align}
			\ip{\nabla G(\bfx)}{\bfy - \bfx} 
			&= \ip{\nabla G(\bfx)}{\bfy} - \ip{\nabla G(\bfx)}{\bfx} \notag \\
			&\leq \ip{\lambda(\nabla G(\bfx))}{\lambda(\bfy)} - \ip{\lambda(\nabla G(\bfx))}{\lambda(\bfx)} \label{eq:pseudoconvex_inequality_1} \\
			&= \ip{\nabla f(\lambda(\bfx))}{\lambda(\bfy)} - \ip{\lambda(\nabla G(\bfx))}{\lambda(\bfx)} \notag \\
			&= \ip{\nabla f(\lambda(\bfx))}{\lambda(\bfy) - \lambda(\bfx)} \notag \\
			&\leq 0. \label{eq:pseudoconvex_inequality_2}
		\end{align}
	\end{subequations}
	Here, the inequality \eqref{eq:pseudoconvex_inequality_1} is due to the von Neumann-type inequality (\autoref{prop:vonNeumann_inequality}) and the first-order characterization of differentiable quasiconvex functions yields the second inequality.
	
	\gap
	
	$(i)$ Suppose $f$ is pseudoconvex on $Q$. Then $f$ is quasiconvex. Let $\bfx, \bfy \in C$ satisfy $G(\bfy) < G(\bfx)$. Since $f(\lambda(\bfy)) < f(\lambda(\bfx))$, the pseudoconvexity of $f$ implies that the inequality \eqref{eq:pseudoconvex_inequality_2} is strict. This shows that $G$ is pseudoconvex.
	
	We now show the converse. Suppose $G$ is pseudoconvex on $C$. Take $u, v \in Q$ such that $f(v) < f(u)$. Choose a permutation matrix $P$ such that $Pu = u^\downarrow$. For a fixed Jordan frame $\{\bfe_1, \bfe_2, \ldots, \bfe_n\}$, define $\bfx, \bfy \in C$ by
	\begin{equation} \label{eq:pseudoconvex_xy}
		\bfx = \sum_{i=1}^{n} [Pu]_i \bfe_i, \quad \bfy = \sum_{i=1}^{n} [Pv]_i \bfe_i.
	\end{equation}
	Since $\lambda(\bfx) = Pu$, we obtain $\displaystyle \nabla G(\bfx) = \sum_{i=1}^{n} [\nabla f(Pu)]_i \bfe_i$ from \autoref{prop:transfer_differentiability}. Since $f$ is symmetric, we have $\nabla f(Pu) = P \nabla f(u)$ by the Chain rule. Also, symmetry of $f$ gives $G(\bfy) = f(v) < f(u) = G(\bfx)$. Thus, $\ip{\nabla G(\bfx)}{\bfy - \bfx} < 0$ by pseudoconvexity of $G$. Now, by the orthonormality of Jordan frames, 
	\begin{align*}
		\ip{\nabla G(\bfx)}{\bfy - \bfx} & = \ip{\sum_{i=1}^{n} [\nabla f(Pu)]_i \bfe_i}{\sum_{j=1}^{n} ([Pv]_j - [Pu]_j) \bfe_j} \\
		& = \sum_{i, j = 1}^{n} [\nabla f(Pu)]_i ([Pv]_j - [Pu]_j) \ip{\bfe_i}{\bfe_j} \\
		& = \sum_{k=1}^{n} [\nabla f(Pu)]_k ([Pv]_k - [Pu]_k) \\
		& = \sum_{k=1}^{n} [P \nabla f(u)]_k [P(v - u)]_k) \\
		& = \ip{P\nabla f(u)}{P(v - u)} \\
		& = \ip{\nabla f(u)}{v - u}.
	\end{align*}
	This shows that $f$ is pseudoconvex on $Q$.
	
	\gap
	
	$(ii)$ Suppose $f$ is strictly pseudoconvex on $Q$. Take $\bfx, \bfy \in C$ such that $\bfx \neq \bfy$ and $G(\bfy) \leq G(\bfx)$. Then we have $f(\lambda(\bfy)) \leq f(\lambda(\bfx))$. We now consider two cases. First, assume $\lambda(\bfx) \neq \lambda(\bfy)$. In this case, the strict pseudoconvexity of $f$ implies that \eqref{eq:pseudoconvex_inequality_2} is strict.
	
	It remains to consider the case $\lambda(\bfx) = \lambda(\bfy)$. We claim that $\nabla G(\bfx)$ and $\bfy$ do not strongly operator commute. Suppose, to the contrary, that they do. Since $f$ is strictly pseudoconvex, \autoref{lem:pseudoconvex_lemma_1} and \cite[Remark 37]{baes} imply that
	\[ [\lambda(\bfx)]_i = [\lambda(\bfx)]_j \iff [\nabla f(\lambda(\bfx))]_i = [\nabla f(\lambda(\bfx))]_j. \]
	By \autoref{lem:gradient_order_quasiconvex}, $\bfx$ and $\nabla G(\bfx)$ strongly operator commute and $\lambda(\nabla G(\bfx))=\nabla f(\lambda(\bfx))$. Therefore \autoref{lem:pseudoconvex_lemma_2} implies that $\bfx$ and $\bfy$ strongly operator commute. Since $\lambda(\bfx) = \lambda(\bfy)$, this gives $\bfx = \bfy$, reaching a contradiction. Hence $\nabla G(\bfx)$ and $\bfy$ do not strongly operator commute, and so
	\[ \ip{\nabla G(\bfx)}{\bfy} < \ip{\lambda(\nabla G(\bfx))}{\lambda(\bfy)}. \]
	Thus the inequality \eqref{eq:pseudoconvex_inequality_1} becomes strict. Hence, in both cases, $G$ is strictly pseudoconvex on $C$.
	
	Conversely, suppose $G$ is strictly pseudoconvex on $C$. Take $u, v \in Q$ such that $u \neq v$ and $f(v) \leq f(u)$. Defining $\bfx, \bfy \in C$ as in \eqref{eq:pseudoconvex_xy}, we readily see that $\bfx \neq \bfy$ and, as $f$ is symmetric, $G(\bfy) \leq G(\bfx)$. It follows that
	\[ \ip{\nabla f(u)}{v - u} = \ip{\nabla G(\bfx)}{\bfy - \bfx} < 0. \]
	Therefore $f$ is strictly pseudoconvex on $Q$.
\end{proof}

\begin{example}
	Consider the spectral open convex set $C = \set{\bfx \in \V}{\tr(\bfx) > -1}$ and the following linear-fractional objective function $G : C \to \R$ defined by
	\[ G(\bfx) = \frac{\tr(\bfx)}{1 + \tr(\bfx)}. \]
	Then $G$ is a spectral function with the associated symmetric function $f : Q \to \R$, where $Q = \set{u \in \Rn}{\sum_{i=1}^{n} u_i > -1}$, given by $f(u) = \phi\big( \sum_{i=1}^{n} u_i \big)$, where $\phi(t) = \frac{t}{1 + t}$ for $t > -1$. Since $\phi$ is strictly increasing on $(-1, \infty)$, it follows that $f$ is pseudoconvex. Thus, by the transfer principle for pseudoconvexity, we conclude that $G$ is pseudoconvex on $C$.
\end{example}

We now consider another notion of generalized convexity frequently used in optimization.

\begin{definition}
	Let $C$ be a convex set in $\V$ and consider a function $G : C \to \R$. Let $\mu > 0$.
	\begin{itemize}
		\item[(a)] $G$ is said to be \emph{$\mu$-strongly convex} on $C$ if a function $G_{\mu} : C \to \R$ defined by 
		\[ G_{\mu}(\bfx) = G(\bfx) - \frac{\mu}{2} \norm{\bfx}^2\] 
		is convex. Equivalently,
		\[ G((1 - \alpha)\bfx + \alpha \bfy) \leq (1 - \alpha)G(\bfx) + \alpha G(\bfy) - \frac{\mu}{2} \alpha(1 - \alpha) \norm{\bfx - \bfy}^2 \]
		holds for all $\bfx, \bfy \in C$ and $\alpha \in [0, 1]$.
		
		\item[(b)] $G$ is said to be \emph{$\mu$-strongly quasiconvex} on $C$ if 
		\[ G((1 - \alpha)\bfx + \alpha \bfy) \leq \max\{ G(\bfx), G(\bfy) \} - \frac{\mu}{2} \alpha(1 - \alpha) \norm{\bfx - \bfy}^2 \]
		holds for all $\bfx, \bfy \in C$ and $\alpha \in [0, 1]$.
	\end{itemize}
\end{definition}

We now describe transfer principles for $\mu$-strong convexity and $\mu$-strong quasiconvexity.

\begin{theorem}
	Let $Q$ be a convex symmetric set in $\Rn$ and $f : Q \to \R$ be a symmetric function. Define $C = \lambda^{-1}(Q)$ and $G : C \to \R$ by $G = f \circ \lambda$. Then, for $\mu > 0$, $G$ is $\mu$-strongly convex on $C$ if and only if $f$ is $\mu$-strongly convex on $Q$.
\end{theorem}

\begin{proof}
	Recall that the set of spectral functions forms a vector space. That is,
	\[ (\alpha_1 G_1 + \alpha_2 G_2)(\bfx) = (\alpha_1 f_1 + \alpha_2 f_2)(\lambda(\bfx)) \]
	for all spectral functions $G_1 = f_1 \circ \lambda$, $G_2 = f_2 \circ \lambda$ and $\alpha_1, \alpha_2 \in \R$. Hence, for each $\bfx \in C$,
	\[ G_\mu(\bfx) = G(\bfx) - \frac{\mu}{2} \norm{\bfx}^2 = f(\lambda(\bfx)) - \frac{\mu}{2} \norm{\lambda(\bfx)}^2 = f_\mu(\lambda(\bfx)), \]
	where the second equality is due to the fact that $\lambda$ preserves the norm. This shows that $G_\mu$ is a spectral function with the associated symmetric function $f_\mu$ given by $f_\mu(u) = f(u) - \frac{\mu}{2} \norm{u}^2$. The conclusion now follows from \autoref{thm:convexity}.
\end{proof}

\begin{theorem}
	Let $Q$ be a convex symmetric set in $\Rn$ and $f : Q \to \R$ be a symmetric function. Define $C = \lambda^{-1}(Q)$ and $G : C \to \R$ by $G = f \circ \lambda$. Then, for $\mu > 0$, $G$ is $\mu$-strongly quasiconvex on $C$ if and only if $f$ is $\mu$-strongly quasiconvex on $Q$.
\end{theorem}

\begin{proof}
	We first record a useful inequality. Suppose that $f$ is symmetric and $\mu$-strongly quasiconvex on $Q$. If $u, v \in Q$ and $u \prec v$, then
	\begin{equation} \label{eq:strong_qc_majorization}
		f(u) \leq f(v) - \frac{\mu}{2} \big( \norm{v}^{2} - \norm{u}^{2} \big).
	\end{equation}
	Indeed, by \cite[Theorem 2.B.1]{marshal-olkin-arnold}, $u$ can be obtained from $v$ by successive applications of a finite number of $T$-transformations. That is, there exist $v_0, v_1, \ldots, v_m \in Q$ such that $v_0 = v$, $v_m = u$, and
	\[ v_{k+1} = (1 - \tau_{k}) v_{k} + \tau_{k} Q_{k} v_{k} \]
	for $\tau_{k} \in [0, 1]$ and some transposition matrix $Q_{k}$, for each $k = 0, 1, \ldots, m-1$. Note that each $v_k$ belongs to $Q$ as $Q$ is convex and symmetric. Also, we see that $\norm{Q_k v_k} = \norm{v_k}$. After some algebraic manipulation, it follows that
	\[ \norm{v_{k+1}}^{2} = \norm{v_{k}}^{2} - \tau_{k} (1 - \tau_{k}) \norm{v_{k} - Q_{k}v_{k}}^{2}. \]
	Since $f$ is symmetric, $f(Q_{k} v_{k}) = f(v_{k})$. Hence, by $\mu$-strong quasiconvexity,
	\begin{align*}
		f(v_{k+1}) &\leq f(v_{k}) - \frac{\mu}{2} \tau_{k} (1 - \tau_{k}) \norm{v_{k} - Q_{k} v_{k}}^{2} \\
		&= f(v_{k}) - \frac{\mu}{2} \big( \norm{v_{k}}^{2} - \norm{v_{k+1}}^{2} \big).
	\end{align*}
	Summing these inequalities for $k = 0, 1, \ldots, m-1$ yields
	\[ f(u) \leq f(v) - \frac{\mu}{2} \big( \norm{v}^{2} - \norm{u}^{2}\big), \]
	which verifies \eqref{eq:strong_qc_majorization}.
	
	Now suppose that $f$ is $\mu$-strongly quasiconvex on $Q$. Take $\bfx, \bfy \in C$ and $\alpha \in [0, 1]$. Define
	\[ \bfz = (1 - \alpha) \bfx + \alpha \bfy \quad \text{and} \quad
	w = (1 - \alpha) \lambda(\bfx) + \alpha \lambda(\bfy). \]
	Since $\lambda$ is norm-preserving, a direct calculation gives
	\[ \norm{w}^{2} - \norm{\lambda(\bfz)}^{2} = \alpha (1 - \alpha) \big( \norm{\bfx - \bfy}^{2} - \norm{\lambda(\bfx) - \lambda(\bfy)}^{2} \big). \]
	Now, we have $\lambda(\bfz) \prec w$ by Ky-Fan inequality. Thus, from \eqref{eq:strong_qc_majorization}, we obtain
	\[ f(\lambda(\bfz)) \leq f(w) - \frac{\mu}{2} \big( \norm{w}^{2} - \norm{\lambda(\bfz)}^{2} \big). \]
	Moreover, by $\mu$-strong quasiconvexity of $f$,
	\[ f(w) \leq \max \{ f(\lambda(\bfx)), f(\lambda(\bfy)) \} - \frac{\mu}{2} \alpha (1 - \alpha) \norm{\lambda(\bfx) - \lambda(\bfy)}^{2}. \]
	Combining the above identities and inequalities, we see that
	\begin{align*}
		G(\bfz) &= f(\lambda(\bfz)) \\
		&\leq \max \{ f(\lambda(\bfx)), f(\lambda(\bfy)) \} - \frac{\mu}{2} \left[ \alpha (1 - \alpha) \norm{\lambda(\bfx) - \lambda(\bfy)}^{2} + \norm{w}^{2} - \norm{\lambda(\bfz)}^{2} \right] \\
		&= \max \{ G(\bfx), G(\bfy) \} - \frac{\mu}{2} \alpha (1 - \alpha) \norm{\bfx - \bfy}^{2}.
	\end{align*}
	Hence $G$ is $\mu$-strongly quasiconvex on $C$.
	
	Conversely, suppose that $G$ is $\mu$-strongly quasiconvex on $C$. Take $u, v \in Q$ and $\alpha \in [0, 1]$. Let $w = (1 - \alpha)u + \alpha v$. Fix a Jordan frame $\{\bfe_1, \bfe_2, \ldots, \bfe_n\}$ and define $\bfx, \bfy, \bfz \in C$ as in \eqref{eq:quasiconvex_xyz}. By the symmetry of $f$, we have $G(\bfx) = f(u)$, $G(\bfy) = f(v)$, and $G(\bfz) = f(w)$. Also, by orthonormality of a Jordan frame, we get
	\[ \norm{\bfx - \bfy}^2 = \bigg\Vert \sum_{i=1}^{n} (u_i - v_i) \bfe_i \bigg\Vert^2 = \sum_{i=1}^{n} (u_i - v_i)^2 \norm{\bfe_i}^2 = \norm{u - v}^2. \]
	
	Thus the $\mu$-strong quasiconvexity of $G$ gives
	\begin{align*}
		f(w) = G(\bfz) &\leq \max\{G(\bfx), G(\bfy)\} - \frac{\mu}{2} \alpha (1 - \alpha) \norm{\bfx - \bfy}^{2} \\
		&= \max\{f(u), f(v)\} - \frac{\mu}{2} \alpha (1 - \alpha) \norm{u - v}^{2}.
	\end{align*}
	Therefore $f$ is $\mu$-strongly quasiconvex on $Q$.
\end{proof}

\section{Concluding remarks}

In this paper, we established transfer principles for several generalized convexity notions, together with transfer principles for strong convexity and strong quasiconvexity of spectral functions on Euclidean Jordan algebras.

\gap
A natural direction for further research is to examine whether analogous transfer principles remain valid in more abstract settings. In this regard, the theory of Fan-Theobald-von Neumann systems (FTvN systems, for short) and semi-FTvN systems, developed and studied in \cite{gowda_FTvNsystem, gowda-jeong_FTvNsystem1, jeong-gowda_FTvNsystem2, gowda-sossa_semiFTvNsystem}, provides a promising framework for such an extension. These systems were introduced as an axiomatic setting for studying optimization problems involving the sum, or a suitable combination, of an objective function and a spectral function over a spectral set.

\gap
A \emph{semi-FTvN} system is a triple $(\V, \W, \lambda)$, where $\V$ and $\W$ are real inner product spaces and $\lambda : \V \to \W$ is an \emph{eigenvalue map} satisfying
\[ \norm{\lambda(\bfx)} = \norm{\bfx} \quad \text{and} \quad \ip{\bfx}{\bfy} \leq \ip{\lambda(\bfx)}{\lambda(\bfy)} \]
for $\bfx, \bfy \in \V$. In this system, for $\bfu \in \V$, the set $[\bfu] = \set{\bfx \in \V}{\lambda(\bfx) = \lambda(\bfu)}$ is called the \emph{eigenvalue orbit} of $\bfu$. A semi-FTvN system $(\V, \W, \lambda)$ becomes an \emph{FTvN system} provided $\max_{\bfx \in [\bfu]} \ip{\bfc}{\bfx} = \ip{\lambda(\bfc)}{\lambda(\bfu)}$ for any $\bfc, \bfu \in \V$. Since such systems retain an abstract eigenvalue map together with a von Neumann-type inequality, they provide a framework for extending the concepts of spectral sets and spectral functions beyond Euclidean Jordan algebras. It would therefore be interesting to investigate whether the transfer principles can be extended to semi-FTvN systems or FTvN systems, possibly under some additional assumptions.

\section*{Acknowledgment}

This work was supported by the Soongsil University Research Fund (New Faculty Research Support) of 2025 (Grant No.\,202510001162).


\end{document}